\documentclass[letterpaper, 10 pt, conference]{ieeeconf}  % Comment this line out if you need a4paper

\IEEEoverridecommandlockouts                              % This command is only needed if 
\usepackage{amssymb}
\usepackage{amsmath}
\usepackage{bm}
\usepackage{graphicx}
\usepackage{ mathrsfs }
\newtheorem{theorem}{Theorem}

\newtheorem{lemma}[theorem]{Lemma}
\newtheorem{assumption}{Assumption}
\newtheorem{definition}{Definition}

\newtheorem{remark}{Remark}
\usepackage{color}
\usepackage{cite}
\usepackage{textcmds}
\usepackage{subcaption}
\usepackage[font=small]{caption}
\usepackage{multirow}
\usepackage{soul}

\allowdisplaybreaks

\title{\LARGE \bf
Perfect Tracking of Time-Varying Optimum by Extremum Seeking 
}

\author{Cemal Tugrul Yilmaz, Mamadou Diagne and Miroslav Krstic
\thanks{C. T. Yilmaz, M. Diagne and M. Krstic are with Department of Mechanical and Aerospace Engineering, University of California, San Diego, La Jolla, CA, USA.
        {\tt\small cyilmaz@ucsd.edu, mdiagne@ucsd.edu, krstic@ucsd.edu}}
}

\begin{document}

\maketitle
\thispagestyle{empty}
\pagestyle{empty}

%%%%%%%%%%%%%%%%%%%%%%%%%%%%%%%%%%%%%%%%%%%%%%%%%%%%%%%%%%%%%%%%%%%%%%%%%%%%%%%%
\begin{abstract}

This paper introduces extremum seeking (ES) algorithms designed to achieve perfect tracking of arbitrary time-varying extremum. In contrast to classical ES approaches that employ constant frequencies and controller gains, our algorithms leverage time-varying parameters, growing either asymptotically or exponentially, to achieve desired convergence behaviors. Our stability analysis involves state transformation, time-dilation transformation, and Lie bracket averaging. The state transformation is based on the multiplication of the input state by asymptotic or exponential growth functions. The time transformation enables tracking of the extremum as it gradually converges to a constant value when viewed in the dilated time domain. Finally, Lie bracket averaging is applied to the transformed system, ensuring practical uniform stability in the dilated time domain as well as asymptotic or exponential stability of the original system in the original time domain. We validate the feasibility of these designs through numerical simulations.

\end{abstract}

%%%%%%%%%%%%%%%%%%%%%%%%%%%%%%%%%%%%%%%%%%%%%%%%%%%%%%%%%%%%%%%%%%%%%%%%%%%%%%%%

\section{Introduction} \label{intro}

Extremum seeking (ES) is a powerful model-free optimization technique with a rich history of theoretical research \cite{krstic2000stability, scheinker2024100, durr2013lie, liu2012stochastic} and practical applications \cite{scheinker2021extremum, dochain2011extremum, zhang2007source}.
Most schemes in the ES literature assume that the cost function remains constant and aim to guide the input signal  to ultimately  discover  an unknown fixed optimum. However, this assumption is not always valid in practical situations. Real-world processes often face changing external conditions, causing the \emph{a priori} fixed optimal input value to shift over time. For instance, renewable energy systems operate in fluctuating weather conditions and varying energy demand, necessitating adjustments to controller parameters to maximize power generation \cite{krstic2014extremum,moura2013lyapunov}. In mineral processing, maintaining optimal air recovery by adjusting aeration rates is essential \cite{wepener2023extremum}. In $\text{CO}_2$ heat pump systems, the efficiency, represented by the coefficient of performance, varies due to environmental factors, requiring real-time optimization for maximum efficiency \cite{hu2015extremum}.

Several ES designs have emerged for plants exhibiting periodic steady-state outputs with constant optimizers. One such scheme, detailed in \cite{wang2000extremum}, focuses on minimizing the size of the limit cycle by incorporating a block for detection of the amplitude of the limit cycle and adjusting a controller parameter to reach a constant optimum value.
An ES is designed in \cite{guay2007flatness} to track an optimal orbit of a nonlinear dynamical system. This is achieved by exploiting the system's flatness property.
For the plants that exhibit periodic outputs of known periodicity, another ES controller with a moving-average filter is presented in \cite{haring2013extremum}. 
However, the periodic nature of the performance functions in these designs poses a limitation on their implementation in a broader range of applications, such as those found in \cite{krstic2014extremum,moura2013lyapunov,wepener2023extremum,hu2015extremum}.

Various papers have explored the optimization of arbitrary time-varying cost functions
using ES technique. 
Work on time-varying optimizers is pioneered in \cite{krstic2000performance}, where a generalized ES scheme is developed to track optimizers with known dynamics but unknown coefficients, employing the internal model principle.
In \cite{sahneh2012extremum}, a delay-based strategy is introduced to extract the gradient signal when dealing with slowly varying optima, which is later extended to dynamic systems with input constraints in \cite{ye2013extremum}. The authors in \cite{scheinker2012extremum} present  an ES approach aimed at achieving unknown reference tracking and stabilization for a class of unknown nonlinear systems using time-varying nonlinear high-gain feedback. A model-based adaptive ES algorithm is introduced in \cite{moshksar2015model} for a class of unstable nonlinear system with time-varying extremum. The authors in \cite{grushkovskaya2017extremum} and \cite{hazeleger2020extremum} establish results regarding the local and semi-global practical asymptotic stability of the extremum of a dynamic map. While \cite{hazeleger2020extremum} seeks a constant optimizer by ES to optimize time-varying steady-state plant performance,  \cite{grushkovskaya2017extremum} 
proves convergence towards a neighboorhood of a time-varying optimizer by extending the Lie bracket approximation method. The robustness of Lie bracket-based ES schemes with respect to time-varying parameters is investigated in \cite{labar2022iss} within the framework of input-to-state stability (ISS). Additionally, \cite{poveda2021fixed} studies ISS-like properties of fixed-time extremum seeking with a time-varying cost function. In contrast to prior results, which utilize high-frequency dither signals, \cite{michael2023gradient} develops a cooperative ES scheme for tracking moving sources without the need for dither signals. However, it is worth noting that all these papers achieve convergence to a small neighborhood of the time-varying optimum. To the best of the authors' knowledge, the perfect tracking of an unknown time-varying extremum remains unaddressed in the ES literature.

In this paper, we present a major  expansion in  the methodology of the technique introduced in \cite{yilmaz2023presc} and \cite{yilmaz2023exponential} for achieving unbiased convergence to a constant optimum, to address the problem of tracking time-varying optimum.
We develop two distinct ES algorithms for perfect tracking of the time-varying optimum, achieving both asymptotic and exponential convergence. These algorithms are referred to as asymptotic ES and exponential ES, respectively. To achieve these results, we depart from the typical ES approach used to track time-varying extrema, which relies on high-frequency sinusoids and constant high gains. Instead, in our ES scheme, we employ time-varying frequencies and controller gains that grow monotonically over time. In asymptotic ES, this growth occurs asymptotically fast, while in exponential ES, the growth is exponential. To analyze the stability of our approach, we perform a state transformation using a monotonically increasing function. Following this transformation, we apply a time-dilation transformation, which converts time-varying frequencies into constants and causes the derivative of the optimum to tend to zero in the dilated time domain. We then apply Lie bracket averaging to demonstrate the practical uniform stability of the transformed system. This, in turn, implies local asymptotic or exponential stability of the original system in original time domain, depending on the chosen algorithm, and guarantees asymptotic or exponential convergence of the output to the extremum, provided that the gains are properly selected.

This paper is organized as follows: In Section \ref{sec_prelim}, we provide essential stability notions and definitions that are referenced throughout the paper. Section \ref{sec_ps} outlines the problem formulation. The designs for asymptotic and exponential ES are introduced, along with a formal stability analysis, in Sections \ref{AsymES} and \ref{ExpES}, respectively. Section \ref{sec_numer} presents the numerical results, and finally, our paper concludes with Section \ref{sec_conc}.

\textit{Notation:} The $\delta$-neighborhood of a set $\mathcal{S} \subset \mathbb{R}^n$ is denoted by $U_{\delta}^{\mathcal{S}}=\{x \in \mathbb{R}^n : \inf_{z \in \mathcal{S}} |x-z|<\delta\}$. The Lie
bracket of two vector fields $f, g : \mathbb{R}^n \times \mathbb{R} \rightarrow \mathbb{R}^n$ with $f (\cdot,t), g( \cdot,t)$ being continuously differentiable is defined by $[f, g](x, t) := \frac{\partial g(x,t)}{\partial x}f(x,t)-\frac{\partial f(x,t)}{\partial x}g(x,t)$. The notation $e_i$ corresponds to the $i$th unit vector in $\mathbb{R}^n$. $\mathbb{R}^+$ denotes the set of non-negative real numbers.

\section{Preliminaries} \label{sec_prelim}
Consider a control-affine system
{
\setlength{\abovedisplayskip}{4pt}
\setlength{\belowdisplayskip}{4pt}
\begin{align}
    \dot{x}={}f_0(x,t)+\sum_{i=1}^q f_i(x,t) \sqrt{\omega} u_i(\omega t), \label{conaff}
\end{align}}where $x \in \mathcal{D} \subset \mathbb{R}^n$, $t \in [t_0, \infty)$, $t_0 \in \mathbb{R}^+$, 
$\omega>0$, $f_0: \mathcal{D} \times \mathbb{R}^+ \to \mathbb{R}^n$, $f_i: \mathcal{D} \times \mathbb{R}^+ \to \mathbb{R}^n, u_i : \mathbb{R}^+ \to \mathbb{R}$ for $i=1,\dots,q$. We compute the Lie bracket system corresponding to \eqref{conaff} as follows
{
\setlength{\abovedisplayskip}{4pt}
\setlength{\belowdisplayskip}{4pt}
\begin{align}
    \dot{\bar{x}}={}f_0(\bar{x},t)+\frac{1}{T} \sum_{{i=1}\atop{j=i+1}}^{q} [f_i, f_j](\bar{x},t) \int_0^T \int_0^{\sigma}  u_{j}(\sigma) u_{i}(\rho) d\rho d\sigma \label{conafflie}
\end{align}}where $\bar{x}(t_0)=x(t_0)$, $T>0$. 

We recall the following theorem from \cite{durr2013lie}, which  characterizes the notion of practical stability.
\begin{theorem}   \label{LieBracketAvThe}
Consider the system \eqref{conaff} and let the following conditions hold:
\begin{itemize}
    \item $f_i \in \mathcal{C}^2: \mathcal{D} \times \mathbb{R}^+ \to \mathbb{R}^n$ for $i=0,\dots,q$.
    \item The functions $|f_i(x,t)|$, $\left| \frac{\partial f_i(x,t)}{\partial t} \right|$, $\left| \frac{\partial f_i(x,t)}{\partial x} \right|$, $\left| \frac{\partial^2 f_j(x,t)}{\partial t \partial x} \right|$, $\left| \frac{\partial [f_j, f_k](x,t)}{\partial t} \right|$, $\left| \frac{\partial [f_j, f_k](x,t)}{\partial x} \right|$ are bounded on each compact set $x \in \mathcal{K} \subset \mathcal{D}$ uniformly in $t \geq t_0$, for $i=0,\dots,q$, $j=1,\dots,q$, $k=j,\dots,q$.
    \item The functions $u_j$ are continuous $T$-periodic with some $T>0$, and $\int_0^T u_j(\sigma)d\sigma=0$ for $j=1,\dots,q$.
\end{itemize}
If a compact set $\mathcal{S} \subset \mathcal{D}$ is locally (globally) uniformly asymptotically stable for \eqref{conafflie}, then $\mathcal{S}$ is locally (semi-globally) practically uniformly asymptotically stable for \eqref{conaff}.
\end{theorem}

To provide a clear understanding of the concept of practical stability, we present the following definition from \cite{durr2013lie}:
\begin{definition} 
A compact set $\mathcal{S} \subset \mathbb{R}^n $ is said to be locally practically uniformly asymptotically stable for \eqref{conaff} if the following three conditions are satisfied:
\begin{itemize}
    \item Practical Uniform Stability: For any $\epsilon>0$ there exist
$\delta, \omega_0>0$ such that for all $t_0 \in \mathbb{R}^+$ and $\omega> \omega_0$,
\begin{align}
    x(t_0) \in U_{\delta}^{\mathcal{S}} \Rightarrow x(t) \in U_{\epsilon}^{\mathcal{S}}, \quad t \in [t_0, \infty).
\end{align}
    \item $\delta$-Practical Uniform Attractivity:  Let $\delta>0$. For any $\epsilon>0$ there exist $t_1 \geq 0$ and $\omega_0 > 0$ such that for all $t_0 \in \mathbb{R}^+$ and $\omega>\omega_0$,
\begin{align}
    x(t_0) \in U_{\delta}^{\mathcal{S}} \Rightarrow x(t) \in U_{\epsilon}^{\mathcal{S}}, \quad t \in [t_0+t_1, \infty).
\end{align}
    \item Practical Uniform Boundedness:  For any $\delta>0$ there exist $\epsilon>0$ and $\omega_0>0$ such that for all $t_0 \in \mathbb{R}^+$ and $\omega>\omega_0$,
\begin{align}
    x(t_0) \in U_{\delta}^{\mathcal{S}} \Rightarrow x(t) \in U_{\epsilon}^{\mathcal{S}}, \quad t \in [t_0, \infty).
\end{align}
Furthermore, if $\delta$-practical uniform attractivity holds for every $\delta>0$, then the compact set $\mathcal{S}$ is said to be semi-globally practically uniformly asymptotically stable for \eqref{conaff}.
\end{itemize}
\end{definition}

\section{Problem Statement} \label{sec_ps}
We consider the following optimization problem
\begin{align}
    \min_{\theta \in \mathcal{U}} J(\theta, \zeta(t)), \label{Jopt}
\end{align}
where $\theta \in \mathcal{U} \subset \mathbb{R}^n$ is the input, $\zeta(t) \in \mathbb{R}^{l}$ is an unknown time-varying function, $J \in \mathcal{C}^2: \mathcal{U} \times \mathbb{R}^l \to \mathbb{R}$
is an unknown cost function. We make the following assumptions:
\begin{assumption} \label{Ass0}
There exists a unique continuously differentiable function $\pi: \mathbb{R}^l \to \mathbb{R}^n$  such that, the solution of \eqref{Jopt} is given by
   \begin{align}
       \theta^*(t)={}\pi(\zeta(t))=\arg \min_{\theta(t) \in {}\mathcal{U}} J(\theta(t), \zeta(t)),
   \end{align}
which satisfy the following conditions
\begin{align}
        J(\theta^*(t), \zeta(t)) <{}& J(\theta(t), \zeta(t)), \quad \forall \theta(t) \neq \theta^*(t),       \\
        \frac{\partial J(\theta^*(t), \zeta(t)) }{\partial \theta} ={}&0,
\end{align}
for $t \in [t_0, \infty), t_0 \geq 0$.
\end{assumption}
\begin{assumption} \label{Ass1}
There exist $\kappa_1, \kappa_2>0$ such that
\begin{align}
    (\theta-\theta^*(t))^T \frac{\partial J(\theta, \zeta(t))}{\partial \theta}\geq {}&\kappa_1|\theta-\theta^*(t)|^{2}, \label{strongconvex} \\
    \left| \frac{\partial^2 J(\theta, \zeta(t))}{\partial \theta^2} \right| \leq {}& \kappa_2 I. \label{strongconvex2}
\end{align}
\end{assumption}

\begin{assumption} \label{asympextbound}
The time-varying functions ${\zeta}(t)$, $\dot{\zeta}(t)$, ${\theta}^*(t)$, $\dot{\theta}^*(t)$,  $\ddot{\theta}^*(t)$ satisfy the following bound for $t \in [t_0, \infty)$ 
\begin{align}
|{\zeta}(t)|+|\dot{\zeta}(t)|+|{\theta}^*(t)|+|\dot{\theta}^*(t)|+|\ddot{\theta}^*(t)| \leq {}  M_{\theta}, \label{ass3bound} 
\end{align}
 where $M_{\theta}$ is an unknown positive constant. Furthermore, there exists an unknown positive constant $M_J>0$ such that
\begin{align}
&\left| J(\theta, \zeta(t))\right|+\left|\frac{\partial J(\theta, \zeta(t))}{\partial \zeta}\right|+\left|\frac{\partial^2 J(\theta, \zeta(t))}{\partial \theta \partial \zeta }\right| \leq {}  M_J,  \label{ass3bound2} 
\end{align}
for $t \in [t_0,\infty)$ and for $\theta(t) \in U_{\varepsilon}^{\mathcal{H}}$, $\varepsilon>0$, $\theta^*(t) \in \mathcal{H}$, where $\mathcal{H} \subset \mathbb{R}^n$ is a compact set.  
\end{assumption}

Assumption \ref{Ass0} guarantees the existence of a unique minimum of the function $J(\theta, \zeta(t))$ at $\theta(t)=\theta^*(t)$. Assumption 2 requires the cost function to be strongly convex in $\theta$. Assumption \ref{asympextbound} ensures the boundedness of the time-varying functions of $\zeta^*(t)$, $\theta^*(t)$ and their time derivatives, as well as the cost function $J(\theta,\zeta(t))$ and its partial derivatives.
We measure the unknown function $J(\theta, \zeta(t))$ in real time as follows
\begin{align}
    y(t)={}&J(\theta(t), \zeta(t)), \qquad t \in [t_0,\infty), \label{youtput}
\end{align}
in which $y \in \mathbb{R}$ is the output.

Our objective is to develop ES algorithms that utilize output feedback $y(t)$ to achieve perfect tracking of $\theta^*(t)$ by $\theta$ both asymptotically and exponentially, which consequently  minimizes the value of $y(t)$. This objective  is achieved by properly choosing monotonically increasing frequencies and controller gains.  Importantly, the  design is performed without the need for prior knowledge of the optimal input $\theta^*(t)$ or the function $J(\theta, \zeta(t))$.
To provide a visual representation, the ES designs to be introduced are depicted schematically in Fig. \ref{blockdiag}. The time-varying design parameters $\nu(t), \varphi(t), \eta(t)$ are given in Table \ref{Table1}.
\begin{figure}[t]
    \centering
     \includegraphics[width=0.85\linewidth]{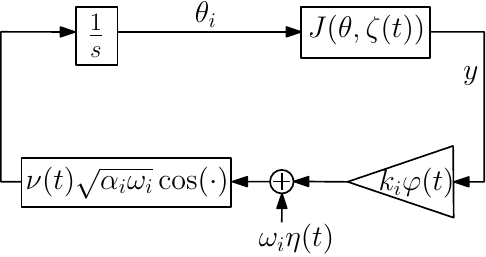}
    \caption{Accelerated ES scheme for the $i$th element $\theta_i$ of $\theta$ with perfect tracking. The scheme relies on a growing frequency signal $\eta(t)$ as well as increasing gains $\nu(t)$ and $\varphi(t)$. For the corresponding functions, refer to Table \ref{Table1}.} 
     \label{blockdiag}
\end{figure}

\begin{table}[t]
\scalebox{1}{
\begin{tabular}{|l |l|}
\hline
\multirow{6}{*}{\textbf{Asymptotic ES }}  & \multirow{2}{*}{$\nu(t)=(1+\beta(t-t_0))^{\frac{m}{r}}$}   \\ [1em] 
 &  \multirow{1}{*}{$\eta(t)=t_0+(1+\beta(t-t_0))^{\frac{r+2}{r}}(t)-1$}  \\ [1em]
 &  \multirow{1}{*}{$\varphi(t)=(1+\beta(t-t_0))^{\frac{2}{r}}(t)$}  \\ [-0.2em]
 &   \\ \hline
 \multirow{5}{*}{\textbf{Exponential ES }}  & \multirow{2}{*}{$\nu(t)=e^{\lambda p(t-t_0)}$}   \\ [1em]
 &  \multirow{1}{*}{$\eta(t)=t_0+e^{2\lambda (t-t_0)}-1$} \\ [0.7em]
 &   $\varphi(t)=e^{2\lambda (t-t_0)}$  \\ [-0.6em]
 &   \\ \hline
\end{tabular}} 
\caption{Time-varying functions used in asymptotic and exponential ES schemes with any $\beta, r, \lambda>0$ and some $m, p>0$.} 
\label{Table1}
\end{table}

In practice, to ensure robustness against noise and numerical issues, the gains $\nu(t)$ and $\varphi(t)$, along with the instantaneous frequency $\omega_i \frac{d\eta(t)}{dt}$, can be constrained to moderately large values that are sufficient for close tracking of the extremum.

\begin{remark}
It is worth noting that there exists flexibility in designing alternative parameters, $\nu(t), \eta(t),$ and $\varphi(t)$, which are capable of achieving the desired convergence results. The choices of the parameters in Table \ref{Table1} prioritize conservative growth in adaptation rate and frequency.
\end{remark}

\section{Asymptotic  ES Design} \label{AsymES}
In this section, we develop an ES referred to as asymptotic ES, which is designed to achieve perfect asymptotic tracking of the time-varying optimum.
The key elements of our design are the asymptotically growing frequencies and controller gains. 
By carefully selecting a rapid growth rate for the time-varying design parameters, $\varphi(t),$ $\eta(t),$ and $\nu(t),$ we guarantee precise tracking of the time-varying optimum.
The asymptotic convergence result is stated in the following theorem.
\begin{theorem} \label{theoremasymp}
Consider the following asymptotic ES design
\begin{align}
     \dot{\theta}={}&\xi^{m}(t) \sum_{i=1}^{n}  \sqrt{\alpha_i \omega_i}  e_i  \nonumber \\
     &\times \cos\Big(\omega_i (t_0+\xi^{r+2}(t)-1)+k_i \xi^2(t)y\Big), \label{ESwithasympt} 
\end{align}
with
\begin{align}
    \xi(t)={}\left(1+\beta(t-t_0)\right)^{\frac{1}{r}}, \qquad t \in [t_0, \infty), \label{xidotdyn}
\end{align}
where $\omega_i = \omega \hat{\omega}_i$ such that $\hat{\omega}_i \neq \hat{\omega}_j$ $\forall i \neq j$, $ t_0 \geq 0$, and
 $\beta, r, k_i, \alpha_i>0$, $ i=1,\dots,n$ 
under  Assumptions \ref{Ass0}--\ref{asympextbound}. There exists $\omega^*>0$
such that for all $\omega > \omega^*$, the following results hold:
\begin{itemize}
\item If $\theta^*$ is constant, i.e., $\dot{\theta}^*(t) \equiv 0$, $\theta(t)$ asymptotically converges to $\theta^*$ for $-\frac{r}{2} \leq m\leq 1$ and $k_i \alpha_i>2(r+2)^2\beta^3/(r^3\kappa_1)$, $ i=1,\dots,n$, 
\item If $\theta^*(t)$ is time-varying, $\theta(t)$ asymptotically converges to $\theta^*(t)$ for $\frac{1}{2}<m\leq 1$ and $k_i \alpha_i>2(r+2)^2(2mr-r+1)\beta^3/(r^3\kappa_1)$, $ i=1,\dots,n$.  
\end{itemize}

\end{theorem}

\begin{proof} Let us proceed through the proof step by step.

\textbf{Step 1: State transformation.} Define the error state 
\begin{align}
    \tilde{\theta}(t)={}\theta(t)-\theta^*(t),
\end{align}
which obeys the following dynamics
\begin{align}
    \dot{\tilde{\theta}}={}&-\dot{\theta}^*(t)+\xi^{m}(t)  \sum_{i=1}^n \sqrt{\alpha_i \omega_i} e_i\cos\Big(\omega_i (t_0+\xi^{r+2}(t)-1)  \nonumber \\
    &+k_i \xi^2(t) J(\tilde{\theta}+\theta^*(t),\zeta(t))\Big), \label{asymthetildedot}
\end{align}
with the help of  \eqref{youtput}. Considering the following transformation
{
\setlength{\abovedisplayskip}{3pt}
\setlength{\belowdisplayskip}{3pt}
\begin{align}
     \tilde{\theta}_f={}\xi(t) \tilde{\theta}, \label{transthf}
\end{align}}which transforms \eqref{asymthetildedot} to
\begin{align}
    \dot{\tilde{\theta}}_f={}&-\xi(t)\dot{\theta}^*(t)+\frac{\beta}{r\xi^{r}(t)} \tilde{\theta}_f+\xi^{m+1} (t) \sum_{i=1}^n \sqrt{\alpha_i \omega_i} e_i \nonumber \\
    &\times \cos\Big(\omega (t_0+\xi^{r+2}(t)-1)+k_i \xi^2(t) J_f(\tilde{\theta}_f,t)\Big), \label{asymthetildefdot}
\end{align}
where
\begin{align}
    J_f(\tilde{\theta}_f,t)={}J(\tilde{\theta}_f/\xi(t)+\theta^*(t),\zeta(t)),  \label{Jfdef}
\end{align}
we rewrite \eqref{asymthetildefdot} by expanding the cosine term
as follows
\begin{align}
    \dot{\tilde{\theta}}_f={}&-\xi(t) \dot{\theta}^*(t)+\frac{\beta}{r\xi^{r}(t)} \tilde{\theta}_f  \nonumber \\
    &+\xi^{m+1}(t)  \sum_{i=1}^n \sqrt{\alpha_i \omega_i} e_i \cos\Big(k_i \xi^2(t)J_f(\tilde{\theta}_f,t)\Big) \nonumber \\
    &\times \cos\Big(\omega_i (t_0+\xi^{r+2}(t)-1)\Big) \nonumber \\
    &-\xi^{m+1}(t)  \sum_{i=1}^n \sqrt{\alpha_i \omega_i} e_i \sin\Big(k_i \xi^2(t)J_f(\tilde{\theta}_f,t)\Big) \nonumber \\
    &\times \sin\Big(\omega_i (t_0+\xi^{r+2}(t)-1)\Big). \label{asymthetildedotexp}
\end{align}

\textbf{Step 2: Time transformation.} To carry out our analysis, we introduce the following time dilation and contraction transformations
 {
\setlength{\abovedisplayskip}{5pt}
\setlength{\belowdisplayskip}{5pt}
\begin{align}
    \tau={}&t_0+\xi^{r+2}(t)-1, \nonumber \\
    ={}&t_0+(1+\beta(t-t_0))^{\frac{r+2}{r}}-1, \quad \tau \in [t_0, \infty),  \label{timetransasym1} \\
    t={}&t_0+\beta^{-1}(\tau-t_0+1)^{\frac{r}{r+2}}-\beta^{-1} \label{timetransasym2}
\end{align}}and using the fact that
\begin{align}
    \frac{d \tau}{dt}
    ={}&\frac{r+2}{r}\beta  (\tau-t_0+1)^{\frac{2}{r+2}}, \label{dtaudtt}
\end{align}
we express \eqref{asymthetildedotexp} in the dilated $\tau$-domain as follows
{
\setlength{\abovedisplayskip}{3pt}
\setlength{\belowdisplayskip}{3pt}
\begin{align}
    \frac{d \tilde{\theta}_f}{d \tau}={} &b_0(\tilde{\theta}_f,\tau)+ \sum_{i=1}^n  b_{c,i}(\tilde{\theta}_f,\tau) \sqrt{\omega_i} \cos(\omega_i \tau ) \nonumber \\
    &- \sum_{i=1}^n  b_{s,i}(\tilde{\theta}_f,\tau) \sqrt{\omega_i} \sin(\omega_i \tau ), \label{asymthetildedotexptau}
\end{align}}
where
\begin{align}
    b_0(\tilde{\theta}_f,\tau)={}&- (\tau-t_0+1)^{\frac{1}{r+2}} \frac{d \theta^*_{\tau}(\tau)}{d\tau}\nonumber \\
    &+\frac{1}{(r+2) (\tau-t_0+1)} \tilde{\theta}_f, \\
    b_{c,i}(\tilde{\theta}_f,\tau)={}&\frac{r\sqrt{\alpha_i }}{(r+2)\beta } (\tau-t_0+1)^{\frac{m-1}{r+2}} e_i  \nonumber \\
    &\times \cos\Big(k_i (\tau-t_0+1)^{\frac{2}{r+2}} J_{f,{\tau}}(\tilde{\theta}_f,\tau)\Big), \\
    b_{s,i}(\tilde{\theta}_f,\tau)={}&\frac{r\sqrt{\alpha_i }}{(r+2)\beta } (\tau-t_0+1)^{\frac{m-1}{r+2}} e_i  \nonumber \\
    &\times \sin\Big(k_i (\tau-t_0+1)^{\frac{2}{r+2}} J_{f,{\tau}}(\tilde{\theta}_f,\tau)\Big)
\end{align}
with 
\begin{align}
    \theta^*_{\tau}(\tau)={}&\theta^*(t_0+\beta^{-1}(\tau-t_0+1)^{\frac{r}{r+2}}-\beta^{-1}), \\
    J_{f,\tau}(\tilde{\theta}_f,\tau)={}&J_f(\tilde{\theta}_f,t_0+\beta^{-1}(\tau-t_0+1)^{\frac{r}{r+2}}-\beta^{-1}). \label{jftsubau}
\end{align}
Note that after the time transformation from $t$ to $\tau$, the
system \eqref{asymthetildedotexptau} has the form of \eqref{conaff} and consequently, Theorem 1 is directly applicable.

\textbf{Step 3: Feasibility analysis of \eqref{asymthetildedotexptau} for averaging.} 
Before performing Lie bracket averaging, we first 
show that \eqref{asymthetildedotexptau} satisfies the boundedness assumption stated in Theorem 1. Let $\mathcal{M} \subset \mathcal{U} $ be a compact set and
consider the bound \eqref{ass3bound} in Assumption \ref{asympextbound}. Then, we get
\begin{align}
    (\tau-t_0+1)^{\frac{1}{r+2}}\left| \frac{d \theta^*_{\tau}(\tau)}{d\tau}\right| \leq {}& \frac{rM_{\theta}}{(r+2)\beta (\tau-t_0+1)^{\frac{1}{r+2}}}, \label{ineq1}\\
    (\tau-t_0+1)^{\frac{1}{r+2}}\left| \frac{d^2 \theta^*_{\tau}(\tau)}{d\tau^2}\right| \leq {}& \frac{r^2M_{\theta}}{(r+2)^2\beta^2 (\tau-t_0+1)^{\frac{3}{r+2}}}, \label{ineq2}
\end{align}
using \eqref{dtaudtt}.
From  inequalities \eqref{ineq1}, \eqref{ineq2},
we obtain the boundedness of $|b_0(\tilde{\theta}_f,\tau)|$, $\big|\frac{\partial b_0(\tilde{\theta}_f,\tau)}{\partial \tilde{\theta}_f}\big|$, $\big|\frac{\partial b_0(\tilde{\theta}_f,\tau)}{\partial \tau}\big|$ for $(\tau,\tilde{\theta}_f) \in  [t_0, \infty) \times \mathcal{M}$. In light of the strong convexity property in Assumption \ref{Ass1}, we obtain the following bound (see. \cite[pp. 461]{boyd2004convex})
\begin{align}
    \left| \frac{\partial  J(\theta, \zeta(t))}{\partial \theta}\right| \leq 2 \kappa_2 |\theta-\theta^*(t)|= \frac{2\kappa_2}{\xi(t)} |\tilde{\theta}_f|. \label{boundgradient}
\end{align}
Again, considering Assumptions \ref{Ass1}, \ref{asympextbound}, \eqref{boundgradient}, and noting from \eqref{transthf}, \eqref{timetransasym1} that
{
\setlength{\abovedisplayskip}{5pt}
\setlength{\belowdisplayskip}{5pt}
\begin{align}
    \frac{d {\theta}}{d \tilde{\theta}_f}={}\frac{1}{\xi(t)}=\frac{1}{(\tau-t_0+1)^{\frac{1}{r+2}}}, \label{dthdthf}
\end{align}}we obtain the following bound
{
\setlength{\abovedisplayskip}{3pt}
\setlength{\belowdisplayskip}{3pt}
\begin{align}
    &(\tau-t_0+1)^{\frac{2}{r+2}}\left(\left| \frac{\partial J_{f,\tau}(\tilde{\theta}_f,\tau)}{\partial \tilde{\theta}_f}\right| +\left|\frac{\partial^2 J_{f,\tau}(\tilde{\theta}_f,\tau)}{\partial \tilde{\theta}_f^2} \right|\right) \nonumber \\
    & \hspace{2.7cm} = {} \xi(t) \left| \frac{\partial J(\theta,\zeta(t))}{\partial \theta} \right|+\left|\frac{\partial^2 J(\theta,\zeta(t))}{\partial \theta^2} \right|,  \nonumber \\
    & \hspace{2.7cm} \leq {}2\kappa_2 |\tilde{\theta}_f|+\kappa_2, \label{taukappa2}
\end{align}}as well as
\begin{align}
    &\frac{d(\tau-t_0+1)^{\frac{2}{r+2}}}{d\tau}\left( |J_{f,\tau}(\tilde{\theta}_f,\tau)|+ \left| \frac{\partial J_{f,\tau}(\tilde{\theta}_f,\tau)}{\partial \tilde{\theta}_f}\right|\right) \nonumber \\
    & \hspace{1.1cm} ={}\frac{2}{(r+2)\xi^{r}(t)} \left(|J(\theta,\zeta(t))| +\frac{1}{\xi(t)}\left| \frac{\partial J(\theta,\zeta(t))}{\partial \theta} \right|\right),\nonumber \\
    & \hspace{1.1cm} \leq {} \frac{2}{(r+2)\xi^{r}(t)} M_J, \label{boundJxi}
\end{align}
for $(t, \tau,\tilde{\theta}_f) \in [t_0, \infty) \times [t_0, \infty) \times \mathcal{M}$.  
In addition, from \eqref{strongconvex2}--\eqref{ass3bound2}, \eqref{dtaudtt}, \eqref{boundgradient}, \eqref{dthdthf}, the following holds
\begin{align}
    &(\tau-t_0+1)^{\frac{2}{r+2}} \left(\left| \frac{\partial J_{f,\tau}(\tilde{\theta}_f,\tau)}{\partial \tau} \right|+\left| \frac{\partial^2 J_{f,\tau}(\tilde{\theta}_f,\tau)}{\partial \tilde{\theta}_f \partial \tau } \right| \right) \nonumber \\
    & = {} \frac{r}{(r+2)\beta } \Bigg( \left| \frac{\partial J(\theta,\zeta(t))}{\partial \zeta(t)} \dot{\zeta}(t) \right|+\Bigg| \frac{\partial J(\theta,\zeta(t))}{\partial \theta} \nonumber \\
    &\hspace{0.45cm} \times \left(-\frac{\tilde{\theta}_f \dot{\xi}(t)}{{\xi}^2(t)}+\dot{\theta}^*(t)\right) \Bigg| + \frac{1}{\xi(t)} \left| \frac{\partial^2 J(\theta,\zeta(t))}{\partial \theta \partial \zeta(t)} \dot{\zeta}(t) \right|\nonumber \\
    &\hspace{0.45cm}+\frac{1}{\xi(t)} \left| \frac{\partial^2 J(\theta,\zeta(t))}{\partial \theta^2} \left(-\frac{\tilde{\theta}_f \dot{\xi}(t)}{{\xi}^2(t)}+\dot{\theta}^*(t)\right)\right| \Bigg), \nonumber \\
    & \leq M_d, \label{partJthtau}
\end{align}
for some $M_d>0$ and for $(t, \tau,\tilde{\theta}_f) \in [t_0, \infty) \times [t_0, \infty) \times \mathcal{M}$.
Taking into account the bounds given through \eqref{taukappa2}--\eqref{partJthtau} and noting that $m \leq 1$, we can arrive at  the boundedness of $|b_{c,i}(\tilde{\theta}_f,\tau)|, |b_{s,i}(\tilde{\theta}_f,\tau)|$, $\big|\frac{\partial b_{c,i}(\tilde{\theta}_f,\tau) }{\partial \theta_f}\big|$, $\big|\frac{\partial b_{s,i}(\tilde{\theta}_f,\tau) }{\partial \theta_f}\big|$, $\big|\frac{\partial b_{c,i}(\tilde{\theta}_f,\tau) }{\partial \tau}\big|$, $\big|\frac{\partial b_{s,i}(\tilde{\theta}_f,\tau) }{\partial \tau}\big|$ , $\big|\frac{\partial^2 b_{c,i}(\tilde{\theta}_f,\tau) }{\partial \tau \partial \tilde{\theta}_f}\big|$ and $\big|\frac{\partial^2 b_{s,i}(\tilde{\theta}_f,\tau) }{\partial \tau \partial \tilde{\theta}_f}\big|$ for $(\tau,\tilde{\theta}_f) \in  [t_0, \infty) \times \mathcal{M}$. Next, the following Lie bracket is computed
{
\setlength{\abovedisplayskip}{3pt}
\setlength{\belowdisplayskip}{3pt}
\begin{align}
    \begin{bmatrix}
        b_{c,i}(\tilde{\theta}_f,\tau) & b_{s,i}(\tilde{\theta}_f,\tau)
    \end{bmatrix}
    ={}&(\tau-t_0+1)^{\frac{2m}{r+2}} \frac{r^2}{(r+2)^2\beta^2}  \nonumber \\
    &\times k_i \alpha_i \frac{\partial J_{f,{\tau}}(\tilde{\theta}_f,\tau)}{\partial \tilde{\theta}_{f,i}}, \label{liebcbs}
\end{align}}which is bounded due to \eqref{taukappa2} and $ m \leq 1$. The boundedness of $\big|\frac{\partial [b_{c,i}(\tilde{\theta}_f,\tau), b_{s,i}(\tilde{\theta}_f,\tau)]}{\partial \tilde{\theta}_{f}}\big|$ and $\big|\frac{\partial [b_{c,i}(\tilde{\theta}_f,\tau), b_{s,i}(\tilde{\theta},\tau)]}{\partial \tau}\big|$ for $(\tau,\tilde{\theta}_f) \in  [t_0, \infty) \times \mathcal{M}$ is obtained by recalling \eqref{taukappa2}--\eqref{partJthtau}. 

\textbf{Step 4: Lie bracket averaging.} We derive the Lie bracket system for \eqref{asymthetildedotexptau} as follows 
\begin{align}
    \frac{d \bar{\theta}_f}{d \tau}={}&b_0(\bar{\theta}_f, \tau)-\frac{1}{2}\sum_{i=1}^n \begin{bmatrix}
     b_{c,i}(\bar{\theta}_f,\tau) & b_{s,i}(\bar{\theta}_f,\tau) \end{bmatrix}  \nonumber \\
    ={}&- (\tau-t_0+1)^{\frac{1}{r+2}}\frac{d \theta^*_{\tau}(\tau)}{d\tau}+\frac{1}{(r+2)(\tau-t_0+1)}\bar{\theta}_f \nonumber \\
    &-(\tau-t_0+1)^{\frac{2m}{r+2}}\frac{r^2}{(r+2)^2\beta^2}   \nonumber \\
    &\times \sum_{i=1}^n \frac{k_i \alpha_i}{2} e_i  \frac{\partial J_{f,{\tau}}(\bar{\theta}_f,\tau)}{\partial \bar{\theta}_{f,i}}. \label{asymthetildedotexptauLie}
\end{align}

\textbf{Step 5: Stability analysis of the Lie bracket system.} We can write \eqref{asymthetildedotexptauLie} in $t$-domain in view of the transformations \eqref{timetransasym1}, \eqref{timetransasym2} as follows
{
\setlength{\abovedisplayskip}{3pt}
\setlength{\belowdisplayskip}{3pt}
\begin{align}
    \dot{\bar{\theta}}_f={}&-\xi(t) \dot{\theta}^*(t)+ \frac{\beta}{r\xi^{r}(t)} \bar{\theta}_f \nonumber \\
    &-\xi^{2m+2}(t) \frac{r^2}{(r+2)^2\beta^2}  \sum_{i=1}^n \frac{k_i \alpha_i}{2}e_i \frac{\partial J_f(\bar{\theta}_f,t)}{\partial \bar{\theta}_{f,i}}. \label{asymthetildedotexptauLiet}
\end{align}}Consider the following Lyapunov function
\begin{align}
    V(\bar{\theta}_f)={}\frac{1}{2}|\bar{\theta}_f|^2. \label{lyapasym}
\end{align}
The computation of  the time derivative of \eqref{lyapasym} along \eqref{asymthetildedotexptauLiet} leads to the following estimates
{
\setlength{\abovedisplayskip}{3pt}
\setlength{\belowdisplayskip}{3pt}
\begin{align}
    \dot{V} \leq {}& \xi(t) |\dot{\theta}^*(t)||\bar{\theta}_f|+\frac{\beta}{r\xi^{r}(t)} |\bar{\theta}_f|^2 -\xi^{2m+2}(t) \frac{r^2 (k\alpha)_{\min}}{2(r+2)^2\beta^2} \nonumber \\
    &\times \sum_{i=1}^n \bar{\theta}_{f,i} \frac{\partial J_f(\bar{\theta}_f,t)}{\partial \bar{\theta}_{f,i}}, \nonumber \\
    \leq {}&- \left( \frac{r^2(k\alpha)_{\min} \kappa_1}{2(r+2)^2\beta^2} \xi^{2m}(t)-\frac{\beta}{r\xi^{r}(t)} \right) |\bar{\theta}_f|^2 \nonumber \\
    &+\xi(t) |\dot{\theta}^*(t)||\bar{\theta}_f|, \label{asymVdotf2}
\end{align}}where $(k\alpha)_{\min}=\min\{k_i \alpha_i\}$ for $i = 1, \dots, n$. To obtain \eqref{asymVdotf2}, we have used    the following property 
\begin{align}
    \bar{\theta}_f^T \frac{\partial J_f(\bar{\theta}_f, t)}{\partial \bar{\theta}_f} \geq {}&\frac{\kappa_1}{\xi^2(t)} |\bar{\theta}_f|^{2}, \label{convprop}
\end{align}
that is derived using \eqref{strongconvex} and \eqref{transthf}.
The stability of \eqref{asymVdotf2} is examined based on the following two cases:

\noindent \underline{\textit{Case 1: Constant $\theta^*$.}} In this case, \eqref{asymVdotf2} reads 
\begin{align}
    \dot{V} \leq {} -2\left( \frac{r^2(k\alpha)_{\min} \kappa_1}{2(r+2)^2\beta^2} -\frac{\beta}{r} \right) \xi^{2m}(t)V, 
\end{align}
and for $m\geq -\frac{r}{2}$, the following estimate holds 
\begin{align}
    V(t) \leq e^{-2\left( \frac{r^2(k\alpha)_{\min} \kappa_1}{2(r+2)^2\beta^2} -\frac{\beta}{r} \right)\int_{t_0}^t \xi^{2m}(\sigma)d\sigma} V(t_0),
\end{align}
by the comparison lemma \cite{khalil2002nonlinear}. For $m = -\frac{r}{2}$, we get $\xi^{2m}(t)=(1+\beta(t-t_0))^{-1}$, and obtain that
\begin{align}
    V(t) \leq  V(t_0) (1+\beta(t-t_0))^{-\frac{2}{\beta}\left( \frac{r^2(k\alpha)_{\min} \kappa_1}{2(r+2)^2\beta^2} -\frac{\beta}{r} \right)},
\end{align}
which decays to zero asymptotically, provided that $(k\alpha)_{\min}>2(r+2)^2\beta^3/(r^3\kappa_1)$. For $m > -\frac{r}{2}$, we write
\begin{align}
    V(t) \leq e^{-2\left( \frac{r^2(k\alpha)_{\min} \kappa_1}{2(r+2)^2\beta^2} -\frac{\beta}{r} \right) \frac{(1+\beta(t-t_0))^{2m/r+1}-1}{\beta(2m/r+1)}} V(t_0),
\end{align}
which decays to zero exponentially, provided that $(k\alpha)_{\min}>2(r+2)^2\beta^3/(r^3\kappa_1)$. This, in turn, implies that  the averaged system \eqref{asymthetildedotexptauLiet} is asymptotically stable. 

\noindent \underline{\textit{Case 2: Time-varying ${\theta}^*(t)$.}} In view of the bound \eqref{ass3bound}, we can rewrite \eqref{asymVdotf2} as follows 
\begin{align}
    \dot{V} \leq {}& - \left( \frac{r^2(k\alpha)_{\min} \kappa_1}{2(r+2)^2\beta^2} \xi^{2m}(t)-\frac{\beta}{r\xi^{r}(t)} \right) |\bar{\theta}_f|^2 +M_{\theta} \xi(t)|\bar{\theta}_f|, \nonumber \\
    \leq {}& -\left( \frac{r^2(k\alpha)_{\min} \kappa_1}{2(r+2)^2\beta^2} -\frac{\beta}{r} \right) \xi^{2m}(t) |\bar{\theta}_f|^2 +M_{\theta} \xi(t)|\bar{\theta}_f|,
    \label{asymVdotf3}
\end{align}
for $m \geq -\frac{r}{2}$. Performing Young's inequalities for the following term 
{
\setlength{\abovedisplayskip}{3pt}
\setlength{\belowdisplayskip}{3pt}
\begin{align}
    M_{\theta} \xi(t)|\bar{\theta}_f| \leq \frac{c_{\xi}}{2} \xi(t)|\bar{\theta}_f|^2+\frac{M_{\theta}^2}{2c_{\xi}}\xi(t),
\end{align}}where $c_{\xi}=\frac{r^2(k\alpha)_{\min} \kappa_1}{2(r+2)^2\beta^2} -\frac{\beta}{r}$, we rewrite \eqref{asymVdotf3} as
\begin{align}
    \dot{V} \leq {}& - \left( \frac{r^2(k\alpha)_{\min} \kappa_1}{2(r+2)^2\beta^2} -\frac{\beta}{r} \right) \xi^{2m}(t)V+\frac{M_{\theta}^2}{2c_{\xi}} \xi(t). \label{Vdotlong2} 
\end{align}
By selecting $m>\frac{1}{2}$ and  applying the result from Lemma \ref{lemVdot}, we conclude that the averaged system \eqref{asymthetildedotexptauLiet} is asymptotically stable given that 
$
    (k\alpha)_{\min} >{}2(r+2)^2(2mr-r+1)\beta^3/(r^3\kappa_1).
$

\textbf{Step 6: Lie bracket averaging theorem.} With the asymptotic stability of the averaged system \eqref{asymthetildedotexptauLiet} proved in Step 5, we conclude from Theorem 1 that the origin of the transformed system \eqref{asymthetildedotexp} is practically uniformly asymptotically stable.

\textbf{Step 7: Convergence to extremum.} Considering the
result in Step 6 and recalling from \eqref{xidotdyn}, \eqref{transthf} that
\begin{align}
    \theta(t)=\theta^*(t)+\frac{1}{(1+\beta(t-t_0))^{\frac{1}{r}}}\tilde{\theta}_f(t),
\end{align}
we conclude the asymptotic convergence of $\theta(t)$ to $\theta^*(t)$.
This implies the convergence of the output $y(t)$ to $J(\theta^*(t), \zeta(t))$ and completes the proof of Theorem \ref{theoremasymp}.
\end{proof}

\section{Exponential ES Design} \label{ExpES}
In this section, we take a further step to accelerate the convergence towards the time-varying optimum. 
We introduce an ES, referred to as exponential ES, which relies on gains and frequencies characterized by exponential growth. By carefully selecting the appropriate growth rate for these signals, we achieve perfect exponential tracking of the bounded time-varying optimum. 
The theorem presented below establishes our exponential convergence result.

\begin{theorem} \label{theoremexp}
Consider the following exponential ES design
\begin{align}
     \dot{\theta}={}&\phi^{p}(t) \sum_{i=1}^n \sqrt{\alpha_i \omega_i}  e_i  \cos\left(\omega_i (t_0+\phi^{2}(t)-1)+k_i \phi^2(t) y\right), \label{ESwithexp} 
\end{align}
with
\begin{align}
    \phi(t)={}e^{\lambda(t-t_0)}, \qquad t \in [t_0, \infty), \label{phidef}
\end{align}
where $\omega_i = \omega \hat{\omega}_i$ such that $\hat{\omega}_i \neq \hat{\omega}_j$ $\forall i \neq j$, $ t_0 \geq 0$, $\lambda>0$,  under  Assumptions \ref{Ass0}--\ref{asympextbound}. There exists $\omega^*>0$
such that for all $\omega > \omega^*$, the followings hold:
\begin{itemize}
\item If $\theta^*$ is constant, i.e., $\dot{\theta}^*(t) \equiv 0$, $\theta(t)$ exponentially converges to $\theta^*$ for $0\leq p \leq 1$ and $k_i \alpha_i>4\lambda^2/\kappa_1$, $ i=1,\dots,n$, 
\item If $\theta^*(t)$ is time-varying, $\theta(t)$ exponentially converges to $\theta^*(t)$ for $\frac{1}{2} < p \leq 1$ and $k_i> 8\lambda^2 p/\kappa_1$, $ i=1,\dots,n$.
\end{itemize}
\end{theorem}

\begin{proof} Let us proceed through the proof step by step.

\textbf{Step 1: State transformation.} Taking the derivative of the error state $\tilde{\theta}(t)={}\theta(t)-\theta^*(t)$ in view of \eqref{ESwithexp} and recalling \eqref{youtput}, we get the following error dynamics
\begin{align}
    \dot{\tilde{\theta}}={}&-\dot{\theta}^*(t)+\phi^{p}(t) \sum_{i=1}^n  \sqrt{\alpha_i \omega_i} e_i \cos\Big(\omega_i (t_0+\phi^{2}(t)-1) \nonumber \\
    &+k_i \phi^2(t) J(\tilde{\theta}+\theta^*(t),\zeta(t))\Big). \label{expthetildedot}
\end{align}
Consider the following transformation
\begin{align}
     \tilde{\theta}_f={}\phi(t) \tilde{\theta}, \label{transthfphi}
\end{align}
which transforms \eqref{expthetildedot} to
{
\setlength{\abovedisplayskip}{3pt}
\setlength{\belowdisplayskip}{3pt}
\begin{align}
    \dot{\tilde{\theta}}_f={}&-\phi(t)\dot{\theta}^*(t)+\lambda  \tilde{\theta}_f+\phi^{p+1} (t) \sum_{i=1}^n  \sqrt{\alpha_i \omega_i}e_i \nonumber \\
    &\times \cos\Big(\omega (t_0+\phi^{2}(t)-1)+k_i \phi^2(t) J_f(\tilde{\theta}_f,t)\Big), \label{expthetildefdot}
\end{align}}with
\begin{align}
    J_f(\tilde{\theta}_f,t)={}J(\tilde{\theta}_f/\phi(t)+\theta^*(t),\zeta(t)).  \label{Jfdef2}
\end{align}

\textbf{Step 2: Time transformation.} Let us perform the following time dilation and contraction transformations
\begin{align}
   \tau_e={}&t_0+\phi^2(t)-1,  \nonumber \\
    ={}&t_0+e^{2\lambda (t-t_0)}-1, \quad \tau_e \in [t_0, \infty), \label{tauedefin} \\
    t={}&t_0+\frac{1}{2\lambda} \ln(\tau_e-t_0+1).
\end{align}
Considering the following fact
\begin{align}
    \frac{d \tau_e}{dt}={}&2\lambda e^{2\lambda  (t-t_0)}={}2\lambda  (\tau_e-t_0+1), \label{dtauedt}
\end{align}
we express \eqref{expthetildefdot} in the dilated $\tau_e$-domain as follows
\begin{align}
    \frac{d \tilde{\theta}_f}{d \tau_e}={}&-(\tau_e-t_0+1)^{\frac{1}{2}} \frac{d \theta^*_{\tau_e}(\tau_e)}{d\tau_e}+\frac{1}{2(\tau_e-t_0+1)}\tilde{\theta}_f  \nonumber \\
    &+ \frac{1}{2\lambda} (\tau_e-t_0+1)^{\frac{p-1}{2}} \sum_{i=1}^n \sqrt{\alpha_i \omega_i} e_i  \nonumber \\
    &\times \cos\Big(\omega \tau_e+k_i (\tau_e-t_0+1) J_{f,\tau_e}(\tilde{\theta}_f,\tau_e)\Big), \label{thetaexpoverte}
\end{align}
with $\theta^*_{\tau_e}(\tau_e)=\theta^*(t_0+1/(2\lambda)\ln(\tau_e-t_0+1))$ and
\begin{align}
    J_{f,\tau_e}(\tilde{\theta}_f,\tau_e)={}&J_f\left(\tilde{\theta}_f,t_0+\frac{1}{2\lambda } \ln(\tau_e-t_0+1)\right). 
\end{align}
We rewrite \eqref{thetaexpoverte} by expanding the cosine term
as follows
\begin{align}
    \frac{d \tilde{\theta}_f}{d \tau_e}={}&-(\tau_e-t_0+1)^{\frac{1}{2}} \frac{d \theta^*_{\tau_e}(\tau_e)}{d\tau_e}+\frac{1}{2(\tau_e-t_0+1)}\tilde{\theta}_f \nonumber \\
    &+\sum_{i=1}^n \frac{\sqrt{\alpha_i}}{2 \lambda}  e_i
    \cos\Big(k_i (\tau_e-t_0+1) J_{f,\tau_e}(\tilde{\theta}_f,\tau_e)\Big)
     \nonumber \\
     &\times (\tau_e-t_0+1)^{\frac{p-1}{2}} \sqrt{\omega_i} \cos(\omega_i \tau_e) \nonumber \\
    &-\sum_{i=1}^n \frac{\sqrt{\alpha_i}}{2\lambda}  e_i 
     \sin\Big(k_i (\tau_e-t_0+1) J_{f,\tau_e}(\tilde{\theta}_f,\tau_e)\Big)  \nonumber \\
     &\times (\tau_e-t_0+1)^{\frac{p-1}{2}} \sqrt{\omega_i} \sin(\omega_i \tau_e). \label{thetaexpoverteexpand}
\end{align}

\textbf{Step 3: Lie bracket averaging.} The feasibility of the error system \eqref{thetaexpoverteexpand} for Lie bracket averaging can be verified analogously to Step 3 in the proof of Theorem \ref{theoremasymp}. For $p \leq 1$, we derive the following average system
\begin{align}
    \frac{d \bar{\theta}_f}{d \tau_e}={}&-(\tau_e-t_0+1)^{\frac{1}{2}} \frac{d \theta^*_{\tau_e}(\tau_e)}{d\tau_e}+\frac{1}{2(\tau_e-t_0+1)}\bar{\theta}_f \nonumber \\
    &-\frac{1}{4\lambda^2}  (\tau_e-t_0+1)^{p} \sum_{i=1}^n \frac{k_i \alpha_i}{2} e_i  \frac{\partial J_{f,{\tau_e}}(\bar{\theta}_f,\tau_e)}{\partial \bar{\theta}_{f,i}}. \label{thetaexpoverteexpandtaue}
\end{align}

\textbf{Step 4: Stability analysis.} To study the stability of \eqref{thetaexpoverteexpandtaue}, we consider the following Lyapunov function
\begin{align}
    V(\bar{\theta}_f)={}\frac{1}{2}|\bar{\theta}_f|^2. \label{lyapexp}
\end{align}
Derivative of \eqref{lyapexp} with respect to $\tau_e$ using \eqref{strongconvex}, \eqref{transthfphi}, \eqref{thetaexpoverteexpandtaue} yields
\begin{align}
    \frac{dV}{d\tau_e} \leq {}&(\tau_e-t_0+1)^{\frac{1}{2}} |\bar{\theta}_f| \left|\frac{d \theta^*_{\tau_e}(\tau_e)}{d\tau_e} \right|+\frac{1}{2(\tau_e-t_0+1)}|\bar{\theta}_f|^2  \nonumber \\
    &-\frac{1}{4 \lambda^2}  (\tau_e-t_0+1)^{p} \sum_{i=1}^n \frac{k_i \alpha_i}{2}  \bar{\theta}_{f,i} \frac{\partial J_{f,{\tau_e}}(\bar{\theta}_f,\tau_e)}{\partial \bar{\theta}_{f,i}}, \nonumber \\
    \leq {}& (\tau_e-t_0+1)^{\frac{1}{2}} |\bar{\theta}_f| \left|\frac{d \theta^*_{\tau_e}(\tau_e)}{d\tau_e} \right|-\bigg(\frac{(k\alpha)_{\min} \kappa_1 }{8\lambda^2} \nonumber \\
    & \times (\tau_e-t_0+1)^{p-1}-\frac{1}{2(\tau_e-t_0+1)}\bigg)|\bar{\theta}_f|^2, \label{dVdtaue}
\end{align}
where $(k\alpha)_{\min}=\min\{k_i \alpha_i\}$ for $i = 1, \dots, n$, we use the property given in \eqref{convprop}, except that $\xi(t)$ is replaced by $\phi(t)$ and we recall from \eqref{tauedefin} that $\phi(t)=(\tau_e-t_0+1)^{\frac{1}{2}}$. The stability of \eqref{dVdtaue} is examined based on the following two cases: 

\noindent \underline{\textit{Case 1: Constant $\theta^*$.}} Noting that $\dot{\theta}^*(t) \equiv 0$, we rewrite \eqref{dVdtaue} as
\begin{align}
    \frac{dV}{d\tau_e} \leq {}&-2\bigg(\frac{(k\alpha)_{\min} \kappa_1 }{8\lambda^2} -\frac{1}{2}\bigg) (\tau_e-t_0+1)^{p-1} V, \label{dVdtauec1}
\end{align}
for $p \geq 0$. In the case where $p=0$, we compute the solution of \eqref{dVdtauec1} by comparison principle, as 
\begin{align}
    V(\tau_e) \leq V(t_0) (\tau_e-t_0+1)^{-2\left(\frac{(k\alpha)_{\min} \kappa_1 }{8\lambda^2} -\frac{1}{2}\right)},
\end{align}
which asymptotically decays to zero, provided that $(k\alpha)_{\min}>4\lambda^2/\kappa_1$. For $p > 0$, we get
\begin{align}
    V(\tau_e) \leq V(t_0) e^{-2\left(\frac{(k\alpha)_{\min} \kappa_1 }{8\lambda^2} -\frac{1}{2}\right) \frac{(\tau_e-t_0+1)^{p}-1}{p}},
\end{align}
which decays to zero exponentially, provided that $(k\alpha)_{\min}>4\lambda^2/\kappa_1$. This, in turn, implies that  the averaged system \eqref{thetaexpoverteexpandtaue} is asymptotically stable.

\noindent \underline{\textit{Case 2: Time-varying ${\theta}^*(t)$.}} Recalling the bound \eqref{ass3bound}, we rewrite \eqref{dVdtaue} as
\begin{align}
    \frac{dV}{d\tau_e} \leq {}& \frac{M_{\theta}}{2\lambda(\tau_e-t_0+1)^{\frac{1}{2}}} |\bar{\theta}_f| -\bigg(\frac{(k\alpha)_{\min} \kappa_1 }{8\lambda^2} -\frac{1}{2}\bigg) \nonumber \\
    &\times (\tau_e-t_0+1)^{p-1} |\bar{\theta}_f|^2, \label{dvdtauec2}
\end{align}
for $p \geq 0$. Performing Young's inequalities for the following term 
\begin{align}
    \frac{M_{\theta}}{2\lambda(\tau_e-t_0+1)^{\frac{1}{2}}} |\bar{\theta}_f| \leq{}& \frac{1}{(\tau_e-t_0+1)^{\frac{1}{2}}} \nonumber \\
    &\times \left( \frac{c_{\phi}}{2} |\bar{\theta}_f|^2+\frac{M_{\theta}^2}{8\lambda^2 c_{\phi}}\right),
\end{align}
where $c_{\phi}=\frac{(k\alpha)_{\min} \kappa_1 }{8\lambda^2} -\frac{1}{2}$, we rewrite \eqref{dvdtauec2} as
\begin{align}
    \frac{dV}{d\tau_e} \leq {}& \frac{M_{\theta}^2}{8\lambda^2 c_{\phi}(\tau_e-t_0+1)^{\frac{1}{2}}} -\bigg(\frac{(k\alpha)_{\min} \kappa_1 }{8\lambda^2} -\frac{1}{2}\bigg) \nonumber \\
    &\times (\tau_e-t_0+1)^{p-1} V.
\end{align}
Choosing $p>\frac{1}{2}$ and applying the result from Lemma \ref{lemVdot}, we conclude that the averaged system \eqref{thetaexpoverteexpandtaue} is asymptotically stable, provided that 
$
    (k\alpha)_{\min} >{}8\lambda^2 p/\kappa_1.
$

\textbf{Step 5: Lie bracket averaging theorem.} 
Taking into account the asymptotic stability of the averaged system \eqref{thetaexpoverteexpandtaue} proved in Step 4, we can deduce, from Theorem 1, that the origin of the transformed system \eqref{thetaexpoverteexpand} in $\tau_e$-domain (equivalent to \eqref{expthetildefdot} in $t$-domain) is practically uniformly asymptotically stable.

\textbf{Step 6: Convergence to extremum.} Considering the
result in Step 5 and recalling from \eqref{phidef}, \eqref{transthfphi} that
\begin{align}
    \theta(t)={}\theta^*(t)+e^{-\lambda (t-t_0)} \tilde{\theta}_f(t),
\end{align}
we conclude the exponential convergence of $\theta(t)$ to $\theta^*(t)$.
This implies the convergence of the output $y(t)$ to $J(\theta^*(t),\zeta(t))$
and completes the proof of Theorem \ref{theoremexp}.
\end{proof}

\begin{figure}[t]
    \centering
    \begin{subfigure}[b]{\linewidth}
        \centering
        \includegraphics[width=\linewidth]{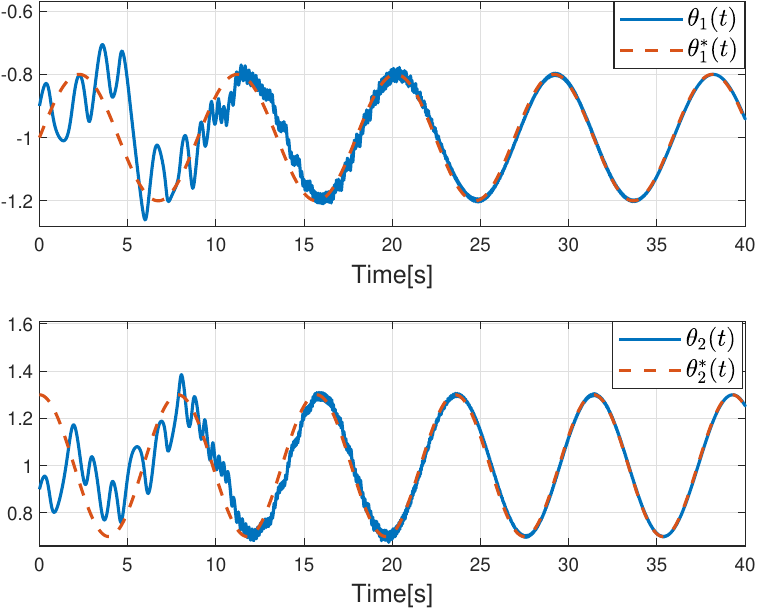}
        \caption{}
        \label{thetafig}
    \end{subfigure}
    \begin{subfigure}[b]{\linewidth}
        \centering
        \includegraphics[width=\linewidth]{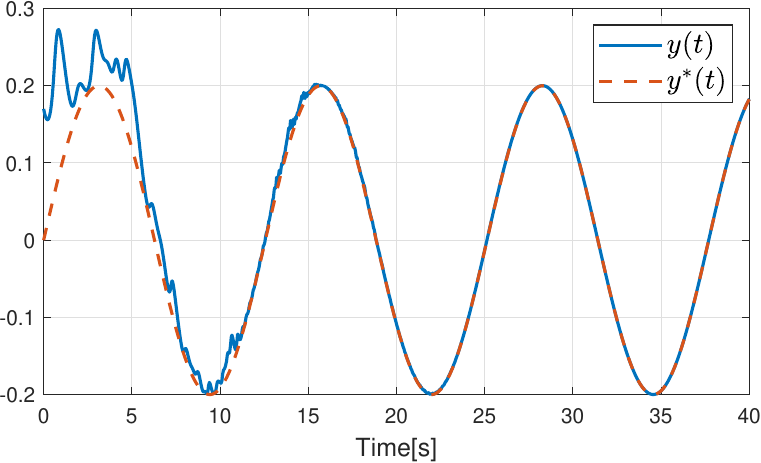}
        \caption{}
        \label{outputfig}
    \end{subfigure}
    \caption{$(a)$ Perfect exponential tracking of the time-varying optimum $\theta_1^*(t), \theta_2^*(t)$ by the inputs $\theta_1, \theta_2$. (b) Exponential convergence of the output $y(t)$ to the time-varying optimum $y^*(t)$.}    
    \label{fig1}
\end{figure}

\section{Numerical Simulation} \label{sec_numer}
In this section, we conduct a numerical simulation to assess the performance of the developed ES algorithms. We consider the following quadratic map
{
\setlength{\abovedisplayskip}{3pt}
\setlength{\belowdisplayskip}{3pt}
\begin{align}
    J(\theta,\theta^*(t))={}&0.2\sin(0.5t)+(\theta_1+1-0.2\sin(0.7t))^2 \nonumber \\
    &+(\theta_2-1-0.3\cos(0.8t))^2,
\end{align}}where $\theta=\begin{bmatrix} \theta_1 & \theta_2 \end{bmatrix}^T \in \mathbb{R}^2$ is the input, $\theta^*(t)=\begin{bmatrix} -1+0.2\sin(0.7t), & 1+0.3\cos(0.8t) \end{bmatrix}^T$ is the optimum input and $y^*(t)=0.2\sin(0.5t)$ is the optimum output. We implement the exponential ES introduced in \eqref{ESwithexp}, employing the following parameters: $p=0.51$, $\alpha_1=0.015$, $\alpha_2=0.02$, $\omega_1=10$, $\omega_2=12$, $k_1=10$, $k_2=11$ and $\lambda=0.1$. The initial conditions are set to $\theta_1(0)=-0.9$, $\theta_2(0)=0.9$. We present the simulation results in Fig. \ref{fig1} and \ref{freqfig}. 
In Figure \ref{thetafig}, we observe that the inputs $\theta_1(t)$ and $\theta_2(t)$ converge toward the optimal inputs $\theta_1^*(t)$ and $\theta_2^*(t)$ exponentially at a rate of $\lambda=0.1$. The exponential growth in frequencies occurs at a rate of $2\lambda=0.2$.
Fig. \ref{outputfig} demonstrates the exponential convergence of the output $y(t)$ toward the time-varying optimal $y^*(t)$. In Fig. \ref{freqfig}, we depict the instantaneous frequencies of \eqref{ESwithexp}, which correspond to $\omega_1 d\phi^2(t)/dt$ and $\omega_2 d\phi^2(t)/dt$, respectively.

\begin{figure}[t]
    \centering
     \includegraphics[width=\linewidth]{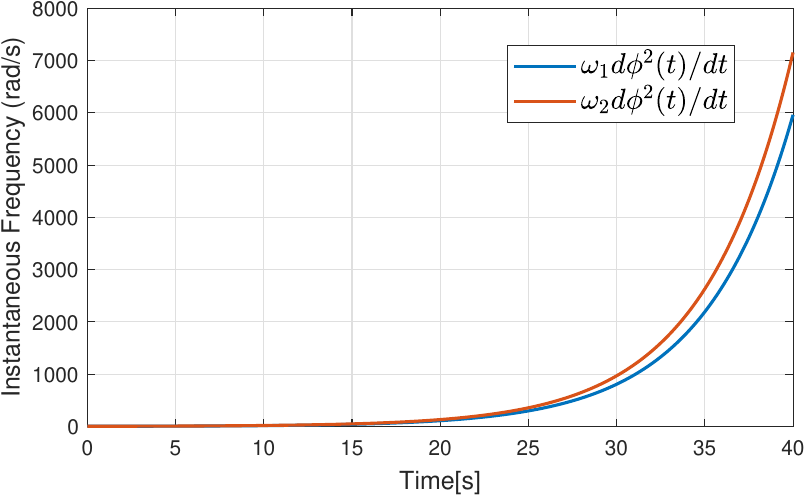}
    \caption{Evolution of the instantaneous frequencies, which correspond to $\omega_1 d\phi^2(t)/dt$ and $\omega_2 d\phi^2(t)/dt$, over time.}    
     \label{freqfig}
\end{figure}

\section{Conclusion} \label{sec_conc}
By focusing on achieving perfect tracking of arbitrary time-varying extremum, this contribution fills  a  gap in the existing ES literature. We introduce two unique ES designs with either asymptotic or exponential convergence, achieved through gains and frequencies characterized by corresponding asymptotic or exponential growth. By carefully tuning the growth rates of these signals, we achieve precise tracking with the desired convergence rate. Our stability analysis relies on state transformation, time-dilation transformation, and Lie bracket averaging techniques. We provide a numerical simulation to illustrate the performance of our exponential ES design in tracking a periodically oscillating extremum.

\section*{APPENDIX}

\subsection{Additional Lemma}
\begin{lemma} \label{lemVdot}
The system
\begin{align}
    \dot{\mathcal{V}}={}  -\varepsilon_a \mu^{m_1}(t) \mathcal{V}+\varepsilon_b \mu^{m_2}(t), 
\end{align}
with 
{
\setlength{\abovedisplayskip}{3pt}
\setlength{\belowdisplayskip}{3pt}
\begin{align}
    \mu(t)={}1+\beta(t-t_0),
\end{align}}for $t \geq t_0$, where $\mathcal{V} \in \mathbb{R}$, $t_0 \geq 0$, $\beta>0$, $\varepsilon_a > \beta (m_1-m_2), \varepsilon_b \in \mathbb{R}$,  $m_1>m_2$, and $m_1 \geq -1$, is asymptotically stable at the origin.
\end{lemma}

\begin{proof}
Consider the following transformation
\begin{align}
    \mathcal{V}_f={}\mu^{m_1-m_2}(t) \mathcal{V}, \label{vftrans}
\end{align}
which obeys the following dynamics
\begin{align}
    \dot{\mathcal{V}}_f={}&\frac{\beta (m_1-m_2) }{\mu(t)} \mathcal{V}_f-\varepsilon_a \xi^{m_1}(t) \mathcal{V}_f +\varepsilon_b \xi^{m_1}(t).
\end{align}
Consider the following Lyapunov function
\begin{align}
    \Upsilon={}\frac{1}{2} \mathcal{V}_f^2, \label{UpsLyap}
\end{align}
whose time derivative yields
\begin{equation}
    \dot{\Upsilon}={}\left(\frac{\beta (m_1-m_2)}{\mu(t)} -\varepsilon_a \mu^{m_1} (t)\right) \mathcal{V}_f^2+\varepsilon_b \mu^{m_1}(t)\mathcal{V}_f.
\end{equation}
Let $\varepsilon_a=\beta (m_1-m_2)+\varepsilon_c$ for any $\varepsilon_c >0$. Then, we get
\begin{align}
    \dot{\Upsilon} \leq {} -\varepsilon_c \mu^{m_1}(t) \mathcal{V}_f^2+\varepsilon_b \mu^{m_1} (t)\mathcal{V}_f, \label{upsdot}
\end{align}
for $m_1 \geq -1$. Performing Young's inequalities, we rewrite \eqref{upsdot} as
{
\setlength{\abovedisplayskip}{3pt}
\setlength{\belowdisplayskip}{3pt}
\begin{align}
    \dot{\Upsilon} \leq {} -\varepsilon_c \mu^{m_1}(t) \Upsilon+\frac{\varepsilon_b^2}{2\varepsilon_c} \mu^{m_1}(t). \label{Upsdotyo}
\end{align}}By comparison principle, we compute from \eqref{Upsdotyo} that
\begin{align}
    \Upsilon(t) &\leq {}e^{-\int_{t_0}^t \varepsilon_c\mu^{m_1}(\sigma)d\sigma }\Upsilon(t_0) \nonumber \\
    &+\frac{\varepsilon_b^2}{2\varepsilon_c} \int_{t_0}^{t} e^{-\int_{\varsigma}^{t} \varepsilon_c\mu^{m_1}(\sigma)d\sigma } \mu^{m_1}(\varsigma)d\varsigma, \nonumber \\
    \leq {}& \begin{cases}
e^{-\frac{\varepsilon_c}{(m_1+1)\beta} \left(\mu^{m_1+1}(t)-1 \right)}\Upsilon(t_0)+\frac{\varepsilon_b^2}{2\varepsilon_c^2},  \hfill \text{for} \, \, m_1>-1, \\
\mu^{-\varepsilon_c/\beta}(t)\Upsilon(t_0)+\frac{\varepsilon_b^2}{2\varepsilon_c^2},   \hfill \text{for} \, \, m_1=-1,
\end{cases}
\end{align}
from which we deduce the stability of $\Upsilon$ and thus,  from \eqref{UpsLyap}, the stability of  $\mathcal{V}_f$. In view of this fact, we conclude the asymptotic stability of $\mathcal{V}$ from \eqref{vftrans}. 
\end{proof}


%%%%%%%%%%%%%%%%%%%%%%%%%%%%%%%%%%%%%%%%%%%%%%%%%%%%%%%%%%%%%%%%%%%%%%%%%%%%%%%%

%%%%%%%%%%%%%%%%%%%%%%%%%%%%%%%%%%%%%%%%%%%%%%%%%%%%%%%%%%%%%%%%%%%%%%%%%%%%%%%%

%%%%%%%%%%%%%%%%%%%%%%%%%%%%%%%%%%%%%%%%%%%%%%%%%%%%%%%%%%%%%%%%%%%%%%%%%%%%%%%%
%\section*{APPENDIX}


%\section*{ACKNOWLEDGMENT}

% \bibliographystyle{plain}
% {
% \bibliography{IEEEabrv,bibliography.bib}}

%%%%%%%%%%%%%%%%%%%%%%%%%%%%%%%%%%%%%%%%%%%%%%%%%%%%%%%%%%%%%%%%%%%%%%%%%%%%%%%%

\end{document}